\documentclass[11pt]{article}

\newcommand{\si}{\sigma}
\newcommand{\al}{\alpha}

\newcommand{\de}{\delta}

\newcommand{\ve}{\varepsilon}

\newcommand{\te}{\theta}
\newcommand{\ka}{\kappa}
\newcommand{\la}{\lambda}

\newcommand{\vp}{\varphi}

\newcommand{\Ga}{\Gamma}
\newcommand{\De}{\Delta}

\newcommand{\Om}{\Omega}
\newcommand{\tl}{\tilde}

\newcommand{\pa}{\partial}

\newcommand{\RR}{\mathbb{R}}

\newcommand{\BS}{\mathbb{S}}
\newcommand{\ii}{\mathrm{i}}

\newcommand{\mc}{\mathcal}
\newcommand{\ms}{\mathscr}

\usepackage[utf8]{inputenc}

\usepackage{geometry}
\usepackage{graphicx}

\usepackage{booktabs} % for much better looking tables
\usepackage{array} % for better arrays (eg matrices) in maths
\usepackage{paralist} % very flexible & customisable lists (eg. enumerate/itemize, etc.)
\usepackage{verbatim} % adds environment for commenting out blocks of text & for better verbatim
\usepackage{caption}
\usepackage{subcaption}

\usepackage{fancyhdr} % This should be set AFTER setting up the page geometry
\usepackage[nottoc,notlof,notlot]{tocbibind}
\usepackage[titles,subfigure]{tocloft}

\usepackage{amsmath,amsfonts,amsthm,mathrsfs,amssymb,cite}
\usepackage[usenames]{color}

\newtheorem{thm}{Theorem}[section]

\theoremstyle{definition}

\numberwithin{equation}{section}

\title{\bf Mathematical Design of A Novel Gesture-based Instruction/Input Device Using Wave Detection }

\author{ Hongyu Liu\thanks{Department of Mathematics, Hong Kong Baptist University, Kowloon Tong, Hong Kong SAR, and HKBU Institute of Research and Continuing Education, Virtual University Park, Shenzhen, P. R. China. Email:  {\tt hongyu.liuip@gmail.com}}
\and Yuliang Wang\thanks{Department of Mathematics, Hong Kong Baptist University, Kowloon Tong, Hong Kong SAR. Email: {\tt yuliang@hkbu.edu.hk}} \and Can Yang\thanks{Department of Mathematics, Hong Kong Baptist University, Kowloon Tong, Hong Kong SAR. Email: {\tt eeyang@hkbu.edu.hk}}}

\date{} % Activate to display a given date or no date (if empty),
\begin{document}
\maketitle

\begin{abstract}

n this paper, we present a conceptual design of a novel gesture-based instruction/input device using wave detection. The device recogonizes/detects gestures from a person and based on which to give the specific orders/inputs to the computing machine that is connected to it. The gestures are modelled as the shapes of some impenetrable or penetrable scatterers from a certain admissible class, called a {\it dictionary}. The device generates time-harmonic point signals for the gesture recognition/detection. It then collects the scattered wave in a relatively small backscattering aperture on a bounded surface containing the point sources. The recognition algorithm consists of two steps and requires only two incident waves of different wavenumbers. The approximate location of the scatterer is first determined by using the measured data at a small wavenumber and the shape of the scatterer is then identified using the computed location of the scatterer and the measured data at a regular wavenumber. We provide the mathematical principle with rigorous justifications underlying the design. Numerical experiments show that the proposed device works effectively and efficiently in some practical scenarios.

\medskip

\medskip



\end{abstract}

\section{Introduction}\label{sect:1}

The technology of gesture computing enables humans to communicate with the machine and interact naturally without any mechanical devices; see Fig.~1 for a schematic illustration. 
\begin{figure}[t]
\centering
\includegraphics[width=10cm]{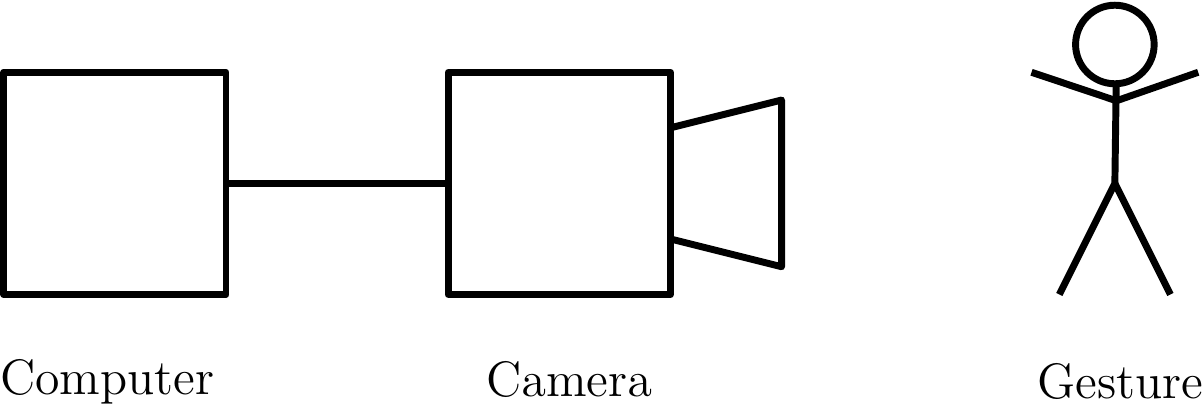}
\caption{Schematic illustration of a traditional gesture recognition device using cameras. }
\end{figure}
Typically, a gesture-based computing technology consists of three major ingredients: the computing machine, the recognition device and the human being who gives orders/instructions to the computer. The recognition device understands human body language, say the hand gesture, and interprets it as specific orders/instructions for the computing machine. Hence, it acts as a bridge between the computing machine and the human being. Using the concept of gesture recognition, it is possible to wave the hand to start the computer and point a finger at the computer screen so that the cursor will move accordingly. This would enrich the communication interfaces between machines and humans other than the conventional text user interfaces or graphical user interfaces, which still limit the majority of input to keyboard and mouse. We refer to \cite{PSH,STTC,EBNBT} for reviews on the state-of-the-art development on gesture computing technology, as well as the interesting Wikipedia article \cite{Wiki} for a comprehensive introduction.   

According to our discussion above, the recognition device plays the key role for a successful implementation of the technology. In nowaday technology, one usually utilizes cameras to capture the images of a person's movements and then the gesture recognition can be conducted with techniques from computer vision and imaging processing. Though there is a large amount of research done in image/video based gesture recognition, there are still many challenges to make the technology more practically useful. In this paper, we present a conceptual design of a novel gesture recognition device, which is different to the conventional image/video-based one. In our design, we use wave signals for the gesture detection and recognition; see Fig.~2 for a schematic illustration of the novel gesture recognition device. The gestures are modelled as shapes of some penetrable or impenetrable scatterers. In order to identity the shapes (and thus the gestures), a transmitter will send out wave signals. The propagation of the  wave field will be interrupted when meeting with the human body performing the gesture, and this generates the so-called scattered wave. Then some receivers will collect the scattered data and from which to identify the shapes of the scattering objects, namely the gestures. We are aware of some very recent engineering development by Google in Project Soli of using radar to identify hand gestures, along with some very interesting technological applications; see \cite{google}.       
\begin{figure}[t]
\centering
\includegraphics[width=10cm]{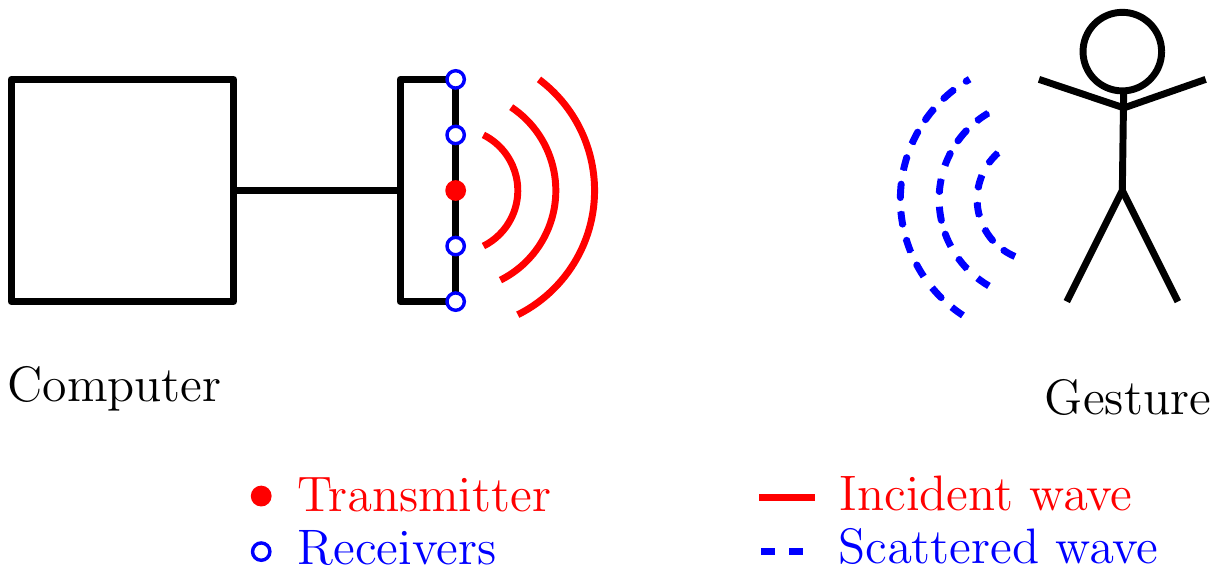}
\caption{Schematic illustration of the novel gesture recognition device using wave scattering. }
\end{figure}

In order to motivate the setup of our novel design, we take a practical scenario as an illustrative example for our subsequent exposition. Let us assume that the gesture recognition device is installed behind the screen of a personal desktop computer, which takes orders/instructions from gestures performed by a person in front of the computer. As mentioned earlier, there are transmitters sending the wave signals and receivers collecting the scattered waves for the gesture recognition. In our design, we only use a single transmitter which generates point wave signals. There is an array of receivers distributed on a bounded surface containing the transmitter. First, we emphasize that from a practical viewpoint, the aperture of the receivers cannot be very large. Second, in order to single out the wave signals for the gesture recognition from possible noises, we shall make use of time-harmonic waves with specified frequencies, namely wavenumbers. Finally, the recognition must be conducted in a timely manner. Indeed, we take account of all these practical factors into our study. We shall make use of two time-harmonic point sources with, respectively, a low wavenumber and a regular wavenumber (in terms of the size of the human body). The backscattering data of a small aperture are measured on a bounded surface containing the point sources. In fact, we can even make use of the phaseless data. The recognition algorithm consists of two steps. In the first step, we determine the location of the scatterer by using the measured data at a small wavenumber. In the second step, with the information of the location of the scatterer, we then identify the shape of the scatterer using the measured data at a regular wavenumber. Our proposed recognition method is computationally very fast, and it is of a totally ``direct" nature without any inversion involved. Hence, our design nicely addresses all the practical concerns mentioned earlier. 

It is pointed out that the mathematical principle of our novel design is a typical inverse scattering problem, where one sends wave fields and collects the scattered wave data, and finally uses the data to identify the unknown scatterer. The study on inverse scattering problems is central to many areas of science and technology, including radar/sonar, geophysical exploration and medical imaging. We refer to \cite{BK12,CK,KG,U} and references therein for related studies on inverse scattering problems. Nevertheless, we would like to emphasize that our study is new to the literature and would meet significant challenges even in the setting of inverse scattering study. As discussed earlier, we shall make use of only two point sources for the identification. The shape identification using a minimum number of scattering measurements is known as the Schiffer's problem (cf. \cite{CK}), and remains to be a longstanding open problem in the literature. There is some significant progress on the Schiffer's problem in determining general polyhedral scatterers \cite{AR,CY,HuLiu,LiuZou1}, and in approximately determining general scatterers \cite{LPRX,Ron1}. Another challenging issue in our study is the phaseless measurement data in a limited aperture. The inverse scattering problems with phaseless data are notoriously difficult in the literature and we refer to \cite{Kli1,Kli2} for some recent progress. A key assumption in our study which can alleviate the mathematical challenges that we are confronting with for the novel gesture recognition device is that the shapes are a priori known to be from an admissible class, called a {\it dictionary}. This is a practically reasonable assumption since one can calibrate the gesture recognition device beforehand by collecting the scattering information of the admissible gestures. It is noted that there are some existing studies in recovering scatterers from an admissible class or dictionary; see \cite{Amm0,Amm1,Amm3,LLZ1}. Among others, one distinctive novelty of our this study lies at the very little scattering information that we use for the identification.

The rest of the paper is organized as follows. In Section 2, we present the mathematical principle study of  the novel gesture recognition device. In Section 3, we conduct extensive numerical experiments to show and verify the effectiveness and efficiency of the proposed recognition algorithm.

\section{Mathematical principle study}

In this section, we present the mathematical modelling and the recognition algorithm for the proposed gesture-based computing device. Following our discussion in Section~\ref{sect:1},  we let
\begin{align}
  u_\ka^{\rm in} (x, t) = \Phi_\ka(x,0)e^{-\mathrm{i}\kappa t}, \quad (x, t)\in\mathbb{R}^3\backslash\{0\}\times\mathbb{R}_+,
\end{align}
be a time-harmonic point source with a wavenumber $\ka$ located at the origin, where
\begin{align*}
\Phi_\ka(x,y) = \frac{e^{\ii \ka \|x-y\|}}{4\pi\|x-y\|}, \ \ x, y\in\mathbb{R}^3\ \ \mbox{and}\ \ x\neq y,
\end{align*}
is the fundamental solution of the PDO, $-\Delta-\ka^2$; namely $(-\De_x - \ka^2)\Phi_\ka(x,y) = \de(x,y)$. By factoring out the time-oscillating part, in what follows, we simply write 
\begin{align}
  \label{eq:2}
  u_\ka^{\rm in} (x) = \Phi_\ka(x,0), \quad x\in\mathbb{R}^3\backslash\{0\},
\end{align}
to denote the time-harmonic point signal located at the origin. 

We model the shape of the body of the person who performs the gestures as a bounded Lipschitz domain $\Omega$. It is assumed that $\Omega$ possesses a connected complement,  $\Om^{\rm ex}=\RR^3 \setminus \bar{\Om}$. Furthermore, as discussed earlier in Section~\ref{sect:1} that the gesture recognition device can be calibrated beforehand, we assume that there exists an {\em admissible class} or a {\em dictionary} of Lipschitz domains,
\begin{equation}\label{eq:dictionary}
\ms{A}=\{D_j\}_{j=1}^N,\ \ N\in\mathbb{N},
\end{equation}
where each $D_j$ is simply connected and contains the origin, such that
\begin{equation}\label{eq:scatterer11}
 \Om = D+z: = \{x+z; x \in D\},\qquad \mbox{where}\ \ D\in\ms{A}\quad\mbox{and}\quad z\in\mathbb{R}^3. 
\end{equation}
It is also assumed that 
\begin{equation}\label{eq:assump1}
\|D_j\| := \max_{x \in D} \|x\| \simeq 1,\quad 1\leq j\leq N,
\end{equation}
and
\begin{equation}\label{eq:assump2}
\|z\|\gg 1,
\end{equation}
where $z\in\mathbb{R}^3$ is the location of $\Omega$ in \eqref{eq:scatterer11}. The assumption \eqref{eq:assump1} means that the Euclidean size of the scatterer $\Omega$ can be calibrated so that we could choose the low frequency in the sense that $2\pi/\ka\gg\|\Omega\|:=\|D\|$ or the regular frequency in the sense that $2\pi/\ka\simeq \|\Omega\|$, of the point source \eqref{eq:2} for the gesture recognition. The assumption \eqref{eq:assump2} means that the person performing the gesture instructions should stay away from the device of a sufficiently large distance. However, it is emphasized that this condition is mainly needed in our theoretical justification of the gesture recognition algorithm in what follows. Indeed, in our numerical experiments, it can be seen that as long as the scatterer $\Omega$ is located away from the point source of a reasonable distance, then the recognition method can work effectively and efficiently to recover the scatterer.

Due to the presence of the scatterer $\Omega$, the propagation of the point wave \eqref{eq:2} will be interrupted, leading to the so-called {\it scattering}. Denote by $u_\ka^\Om(x), x \in \Om^{\rm ex}$ the interrupted/scattered wave field associated with the scatterer $\Om$ and the point source $u_\ka^{\rm in}(x)$ in \eqref{eq:2}. Let $\Ga$ denote a bounded surface containing the origin such that $dist(\Om, \Ga) \gg 1$, where $dist(\Om,\Ga) = \inf_{x \in \Om, y \in \Ga} \|x-y\|$
denotes the distance between $\Om$ and $\Ga$. In our gesture recognition study, $\Gamma$ denotes the measurement surface such that one measures the scattered wave field $u_\ka^\Om|_\Ga$ and from which to recover the shape of the scatterer $\Omega$; that is, $\partial\Omega=z+\partial D$. For a timely recognition, we shall only makes use of two point waves of the form \eqref{eq:2} for two values of $\ka \lesssim 1$. Moreover, from a practical point of view, we shall mainly consider two cases that $\Omega$ is impenetrable being soft or $\Omega$ is penetrable being a medium scatterer. The case that $\Omega$ is soft corresponds to the case that the person wears a certain equipment which prevents the wave field from penetrating inside the body; and the case that $\Omega$ is a medium scatterer correspond to the generic situation where the wave field can pass through the body.   

\subsection{Impenetrable soft scatterer}\label{sect:soundsoft}

We first consider the case where $\Om$ is soft, namely $\Om$ is impenetrable to the wave and there holds the homogeneous boundary condition $u_\ka^{\rm in} + u_\ka^\Om = 0$ on $\pa \Om$. In the frequency domain, the wave scattering is governed by the following PDE system for $u_\ka^\Om$,
\begin{align}
  \label{eq:4}
  \begin{cases}
    (\De + \ka^2) u_\ka^\Om = 0 \quad {\rm in}~\Om^{\rm ex}, \\[5pt]
    u_\ka^\Om = - u_\ka^{\rm in} \quad {\rm on}~\pa \Om^{\rm ex}, \\[5pt]
    \displaystyle{\lim_{r \to \infty} r \left( \frac{\pa}{\pa r} - \ii \ka \right) u_\ka^\Om = 0,}
  \end{cases}
\end{align}
where $r = \|x\|$ and the limit holds uniformly for all $\hat{x}:=x/\|x\| \in \BS^2$, where $\BS^2$ denotes the unit sphere in $\RR^3$. The last limit in \eqref{eq:4} is known as the Sommerfeld tradition condition, which characterizes the decaying property of the scattered wave field away from the scatterer. The scattering system \eqref{eq:4} is well understood \cite{CK,Mcl,Ned}, and there exists a unique solution $u_\ka^\Om\in H^1_{\rm loc}(\mathbb{R}^3\backslash\overline{\Omega})$ which possesses the following asymptotic expansion as $\|x\|\rightarrow+\infty$,
\begin{equation}\label{eq:ffp}
u_\ka^\Om(x)=\frac{e^{\ii\ka\|x\|}}{\|x\|}u_\ka^\infty(\hat x)+\mathcal{O}\left(\frac{1}{\|x\|^2}\right),
\end{equation}
where $u_\ka^\infty$ is known as the far-field pattern for $u_\ka^\Om$, and $\hat x$ is called the observation direction. 

The inverse scattering problem concerning the gesture recognition is to identify $\partial\Omega$ from the measurement of the scattered field on $\Gamma$ due to $\Om$. Noting that $\partial\Omega=z+\partial D$, where $z$ is the location of the scatterer and the shape $D$ is from the dictionary. For the recovery of $\partial \Omega$, it would be advantageous of decoupling the information on $z$ and $D$ in the scattered wave field $u_\ka^\Om|_\Gamma$. This can be done if the incident wave field is a plane wave of the form $e^{\ii\ka d\cdot x}$ with $d\in\mathbb{S}^2$ denoting the incident direction, and one can make use of a certain translation relation; see \cite{LLZ1}. For the current study with the incident wave being a point signal \eqref{eq:2}, our first step is derive a similar relation in order to decouple the information of $z$ and $D$ in the scattered wave field that is measured on $\Gamma$. The main idea arises from the crucial observation that $u^{\rm in}(x)$ near $\Om$ can be approximated by a plane wave propagating in the $\hat{z}$-direction. This is also the reason why we need to assume that the scatterer is located at a reasonably large distance away from the origin which is the location of both the point source and the measurement surface. Nevertheless, we would like to emphasize again that this assumption is mainly required in our theoretical justification, and our numerical experiments show that even without this assumption, the proposed recognition algorithm still works effectively and efficiently.

In what follows, if the incident wave $u_\ka^{\rm in}$ in the scattering system \eqref{eq:4} is replaced by a plane wave $e^{\ii\ka d\cdot x}$, then we write $w_\ka(D, d; x)$, $x\in D^{\rm ex}$, to denote the corresponding scattered wave field, and $w_\ka^\infty(D, d; \hat x)$, $\hat x\in\mathbb{S}^2$, to denote the associated far-field pattern. Then we have the following crucial theorem for the subsequent recovery use. 

\medskip

\begin{thm} \label{thm:point-plane}
Let $\ka\in\mathbb{R}_+$ be fixed. Let $w_\ka(D,\hat{z};x), x \in D^{\rm ex}$ and $w^\infty_\ka(D,\hat{z}; \hat{x}), \hat{x} \in \BS^2$ be, respectively, the scattered field and the far-field pattern corresponding to the sound-soft scatterer $D$ and the plane incident field $w_\ka^{\rm in}(\hat{z};x)=e^{\ii \ka \hat{z} \cdot x}$. Let $u_\ka^\Om$ be the scattered wave field defined in \eqref{eq:4}. Then there holds the following asymptotic expansion,
\begin{equation} \label{eq:23a}
u_\ka^\Om(x)= \frac{e^{\ii \ka \|z\|}}{4 \pi \|z\|} \frac{e^{\ii \ka \|x-z\|}}{\|x-z\|} \left[ w_\ka^\infty(D,\hat{z};\widehat{x-z}) + \mc{O}(\|z\|^{-1}) \right] \left[ 1 + \mc{O}(\|z\|^{-1}) \right]
\end{equation}
as $\|z\| \to \infty$ uniformly for all $\hat{z} \in \BS^2$ and $x \in \Ga$.
\end{thm}

\medskip

\begin{proof}
We first assume that $\partial D$ is $C^2$-continuous and we can make use of the Brakhage-Werner trick for our proof (cf. \cite{CK,Ned}). In this case, it is known that $u_\ka^\Om\in C^2(\Om^{\rm ex})\cap C(\overline{\Om^{\rm ex}})$. We seek the solution to \eqref{eq:4} in the form of combined layer potentials,
\begin{align}
  \label{eq:5}
  u_\ka^\Om(x) = \left( \mc{K}_\ka^\Om + \ii \eta \mc{S}_\ka^\Om \right) [\vp_\ka^\Om] (x), \quad x \in \Om^{\rm ex},
\end{align}
where $\vp_\ka^\Om\in C(\partial \Om)$ and the double-layer potential,
\begin{align*}
  \mc{K}_\ka^\Om[\vp_\ka^\Om](x) := \int_{\pa \Om} \frac{\pa \Phi_\ka(x,y)}{\pa \nu(y)} \vp_\ka^\Om(y) \, {\rm d}s(y),
\end{align*}
and the single-layer potential,
\begin{align*}
  \mc{S}_\ka^\Om[\vp_\ka^\Om](x) := \int_{\pa \Om} \Phi_\ka(x,y) \vp_\ka^\Om(y) \, {\rm d}s(y),
\end{align*}
and $\eta \neq 0$ is a real-valued coupling parameter. By letting $x \to \pa\Om^+$ and using the jump relations for the layer potential operators (cf. \cite{CK,Ned}), together with the boundary condition in \eqref{eq:4}, we obtain the following integral equation for $\vp_\ka^\Om\in C(\partial\Omega)$,
\begin{align*}
  \left( \frac{I}{2} + \mc{K}_\ka^\Om + \ii \eta \mc{S}_\ka^\Om \right) [\vp_\ka^\Om](x) = - u^{\rm in}(x), \quad x \in \pa\Om.
\end{align*}

Using change of variables, we can verify from \eqref{eq:5} that
\begin{align}
  \label{eq:23}
  u_\ka^\Om(x+z) =  \left( \mc{K}_\ka^D + \ii \eta \mc{S}_\ka^D \right) [\vp_\ka^D] (x), \quad x \in D^{\rm ex},
\end{align}
where $\vp_\ka^D(y) = \vp_\ka^\Om(y+z), y \in \pa D$. Applying the jump relations yields
\begin{align}
  \label{eq:25}
    \left( \frac{I}{2} + \mc{K}_\ka^D + \ii \eta \mc{S}_\ka^D \right) [\vp_\ka^D](x) = - u^{\rm in}(x+z), \quad x \in \pa D.
\end{align}
Combining \eqref{eq:23} and \eqref{eq:25} we obtain
\begin{align}
  \label{eq:27}
  u_\ka^\Om(x+z) = - \left( \mc{K}_\ka^D + \ii \eta \mc{S}_\ka^D \right) \left( \frac{I}{2} + \mc{K}_\ka^D + \ii \eta \mc{S}_\ka^D \right)^{-1} u^{\rm in}(x+z), \quad x \in D^{\rm ex}.
\end{align}
Using the asymptotic expansion
\begin{align*}
 \|x+z\| = \|z\| + \hat{z} \cdot x + \mc{O}(\|z\|^{-1}), \quad x \in \pa D,
\end{align*}
which holds uniformly for $\hat{z} \in \BS^2$ as $\|z\| \to \infty$, we can derive
\begin{align}
  \label{eq:26}
  u^{\rm in}(x+z) = \frac{e^{\ii \ka \|z\|}}{4 \pi \|z\|} w_\ka^{\rm in}(\hat{z};x) \left[ 1 + \mc{O}(\|z\|^{-1}) \right], \quad x \in \pa 
D.
\end{align}
Plugging \eqref{eq:26} into \eqref{eq:27}, noting that $\left( \mc{K}_\ka^D + \ii \eta \mc{S}_\ka^D \right) \left( \frac{1}{2} + \mc{K}_\ka^D + \ii \eta \mc{S}_\ka^D \right)^{-1}$ is bounded and
\begin{align*}
w_\ka(D,\hat{z};x) = - \left( \mc{K}_\ka^D + \ii \eta \mc{S}_\ka^D \right) \left( \frac{I}{2} + \mc{K}_\ka^D + \ii \eta \mc{S}_\ka^D \right)^{-1} w_\ka^{\rm in}(\hat{z};x), \quad x \in D^{\rm ex},
\end{align*}
we obtain
\begin{align*}
  u_\ka^\Om(x+z) = \frac{e^{\ii \ka \|z\|}}{4 \pi \|z\|} w_\ka(D,\hat{z};x) \left[ 1 + \mc{O}(\|z\|^{-1}) \right], \quad x \in D^{\rm ex}.
\end{align*}
Using change of variables yields
\begin{align}
  \label{eq:39}
    u_\ka^\Om(x) = \frac{e^{\ii \ka \|z\|}}{4 \pi \|z\|} w_\ka(D,\hat{z};x-z) \left[ 1 + \mc{O}(\|z\|^{-1}) \right], \quad x \in \Om^{\rm ex}.
\end{align}
Plugging the far-field expansion
\begin{align*}
  w_\ka(D,\hat{z};x-z) = \frac{e^{\ii \ka \|x-z\|}}{\|x-z\|} \left[ w_\ka^\infty(D,\hat{z};\widehat{x-z}) + \mc{O}(\|z\|^{-1}) \right] \quad \mbox{as}\ \ \|x-z\| \to \infty,
\end{align*}
into \eqref{eq:39} we finally obtain \eqref{eq:23a}.

It is pointed out that the Brakhage-Werner trick of using combined layer potentials to represent the wave solution is to avoid the interior eigenvalue problem. That is, if $\eta\equiv 0$ and $\ka^2$ is a Dirichlet eigenvalue to $-\Delta$ in $D$, then the integral operator $\frac{I}{2} + \mc{K}_\ka^D + \ii \eta \mc{S}_\ka^D$ in \eqref{eq:27} is no longer invertible. Hence, if one assumes that $\ka^2$ is not a Dirichlet Laplacian eigenvalue to $D$, then one can simply use the double-layer potential to represent the solution     $u_\ka^\Omega$ in \eqref{eq:5}. Particularly, in such a case, by using the mapping properties of the double layer potential operator when $D$ is a Lipschitz domain in \cite{Mcl}, one can follow a completely similar argument as above to show that the theorem holds when $\partial D$ is only Lipschitz-continuous. Then case that $\partial D$ is Lipschitz-continuous and also $\ka^2$ is a Dirichlet Laplacian eigenvalue would need more technical argument by following the techniques in \cite{Mcl}. We shall not give a complete treatment to the last case and instead we shall focus on our study of the gesture recognition. 

The proof is complete. 
\end{proof}

\subsection{Penetrable medium scatterer}\label{sect:medium}

Next we consider the case when the target object $\Omega$ is a penetrable medium scatterer. Let $n_\Omega\in L^\infty(\mathbb{R}^3)$ be a real-valued function such that $supp(n_\Omega-1)\subset\Om$. $m_\Om(x):=n_\Omega(x)-1$, $x\in\Omega$, denotes the refractive index of the medium inside the body $\Omega$. In a similar manner, we let $n_D$ and $m_D$ signify the refractive index functions of the reference scatterer $D$. Recalling that $\Omega=D+z$, one clearly has the following relation,
\[
n_\Omega(x)=n_D(x-z)\ \ \mbox{for}\ \ x\in\mathbb{R}^3. 
\] 

Similar to Theorem \ref{thm:point-plane}, we let $w_\ka^{\rm in}(\hat{z};x) = e^{\ii \ka \hat{z} \cdot x}$ be a plane incident wave and $w_\ka(D,\hat{z};x)$ be the scattered wave due to $D$ and $w_\ka^{\rm in}(\hat{z};x)$. Set
\[
w_\ka^{\rm t}(D,\hat{z}; x):=w_\ka^{\rm in}(\hat{z};x)+w_\ka(D,\hat{z};x),\quad x\in\mathbb{R}^3.
\]
Then the total field $w_\ka^{\rm t}(D,\hat{z}; \cdot)$ satisfies the equation
\begin{align}\label{eq:medsca1}
  (\De + \ka^2 n_D) w_\ka^{\rm t} = 0 \quad {\rm in}~\RR^3.
\end{align}
Equation \eqref{eq:medsca1} together with the Sommerfeld radiation condition on $w_\ka(D,\hat{z};x)$ governs the wave scattering corresponding to the reference medium scatterer $(D, n_D)$ due to time-harmonic plane wave incidence. It follows from the Lippmann-Schwinger equation \cite{CK} that
\begin{align}
  \label{eq:29}
  w_\ka(D,\hat{z};x) = \ka^2 \int_D \Phi_\ka(x,y) m_D(y) w_\ka^{\rm t}(D,\hat{z};y) \, {\rm d}y, \quad x \in \RR^3
\end{align}
where $m_D=n_D-1$ is supported in $D$.

Let $u_\ka^{\rm in}(x)$ be the point-source incident field given by \eqref{eq:2} and $u_\ka^{\rm t}=u_\ka^{\rm in} + u_\ka^\Om$ be the total field due to $(\Om, n_\Omega)$ and $u_\ka^{\rm in}$. Then we have
\begin{align*}
  (\De + \ka^2 n_\Om)u_\ka^{\rm t}(x) = - \de(x),
\end{align*}
where $n_\Om(x)=n_D(x-z)$ and $\de$ denotes the Dirac delta function. It is easy to verify
\begin{align}
  \label{eq:10}
  (\De + \ka^2) u_\ka^\Om(x) =
  \begin{cases}
    0, & x \in \Om^{\rm ex}, \\
    -\ka^2 m_\Om(x) u_\ka^{\rm t}(x), & x \in \Om,
  \end{cases}
\end{align}
where $m_\Om=n_\Om-1$. For any fixed $x \in \RR^3$ let $B$ be a ball centered at the origion such that $B \supset \{x\} \cup \Om$. It follows from Green's formula \cite{CK,Ned} that
\begin{align}
  \label{eq:18}
  u_\ka^\Om(x) &= \int_{\pa B} \frac{\pa u_\ka^\Om(y)}{\pa \nu} \Phi_\ka(x,y) - u_\ka^\Om(y) \frac{\pa \Phi_\ka(x,y)}{\pa \nu(y)} \, {\rm d}s(y) \notag \\
  &- \int_{B} \left[ \De u_\ka^\Om(y) + \ka^2 u_\ka^\Om(y) \right] \Phi_\ka(x,y) \, {\rm d}y.
\end{align}
We can deduce the boundary integral in \eqref{eq:18} is zero by using Green's Theorem in $B' \setminus \bar{B}$ for a ball $B'$ centered at the origion and the Sommerfeld radiation condition for $u_\ka^\Om$. Plugging \eqref{eq:10} into \eqref{eq:18} yields
\begin{align}
  \label{eq:21}
  u_\ka^\Om(x) = \ka^2 \int_\Om \Phi_\ka(x,y) m_\Om(y) u_\ka^{\rm t}(y) \, {\rm d}y, \quad x \in \RR^3.
\end{align}
That is, the Lippman-Schwinger equation remains valid for a point-source incident field located in the exterior of $\Om$.

Similar to Theorem~\ref{thm:point-plane}, we have the following crucial theorem for our subsequent use of recovering $\partial\Omega=\partial D+z$ for $(\Omega, n_\Omega)$. 

\begin{thm}
  Let $u_\ka^\Om, w_\ka(D,\hat{z};\cdot)$ be the scattered wave fields defined above in this section and $w^\infty_\ka(D,\hat{z}; \hat{x}), \hat{x} \in \BS^2$ be the far-field pattern corresponding to $w_\ka(D,\hat{z};\cdot)$, then there holds the asymptotic expansion \eqref{eq:23a}.
\end{thm}
\begin{proof}
Define the operator
\begin{align}
  \label{eq:30}
  \mc{T}_\ka^\Om [v](x) = \ka^2 \int_\Om \Phi_\ka(x,y) m_\Om(y) v(y) \, {\rm d}y.
\end{align}
It is known that $\mc{T}_\ka^\Om$ is a bounded operator from $L^2(\Om)$ to $H^2(\Om)$; we refer to \cite{CK} for more discussion about the mapping properties of this volume integral operator. 
Then we may rewrite \eqref{eq:21} as
\begin{align}
  \label{eq:22}
  (I-\mc{T}_\ka^\Om)u_\ka^\Om = \mc{T}_\ka^\Om u_\ka^{\rm in}.
\end{align}
Similarly we may rewrite \eqref{eq:29} as
\begin{align}
  \label{eq:31}
  (I-\mc{T}_\ka^D) w_\ka(D,\hat{z},\cdot) = \mc{T}_\ka^D w_\ka^{\rm in}(\hat{z};\cdot),
\end{align}
where $\mc{T}_\ka^D$ is defined in the same way as \eqref{eq:30} but with $\Om$ replaced by $D$. 

Introduce the change of variables $\tl{x}=x-z, \tl{u}_\ka^\Om(\tl{x}) = u_\ka^\Om(x)$, then it is easy to verify $\mc{T}_\ka^\Om[u_\ka^\Om](x) = \mc{T}_\ka^D [\tl{u}_\ka^\Om](\tl{x})$. Hence it follows from \eqref{eq:22} that
\begin{align*}
  \tl{u}_\ka^\Om(\tl{x}) = \left[ (I-\mc{T}_\ka^D)^{-1}\mc{T}_\ka^D \right] u_\ka^{\rm in}(x),
\end{align*}
or
\begin{align}
  \label{eq:32}
  \tl{u}_\ka^\Om(x) = \left[ (I-\mc{T}_\ka^D)^{-1}\mc{T}_\ka^D \right] u_\ka^{\rm in}(x+z).
\end{align}
On the other hand it follows from \eqref{eq:31} that
\begin{align}
  \label{eq:33}
  w_\ka(D,\hat{z},x) = \left[ (I-\mc{T}_\ka^D)^{-1} \mc{T}_\ka^D \right] w_\ka^{\rm in}(\hat{z};x).
\end{align}
Since $(I-\mc{T}_\ka^D)^{-1}$ and $\mc{T}_\ka^D$ are bounded (cf. \cite{CK}), combining \eqref{eq:26}, \eqref{eq:32} and \eqref{eq:33} yields
\begin{align*}
  \tl{u}_\ka^\Om(x) = \frac{e^{\ii \ka \|z\|}}{4 \pi \|z\|} w_\ka(D,\hat{z};x) \left[ 1 + \mc{O}(\|z\|^{-1}) \right].
\end{align*}
Reverting the change of variable we conclude
\begin{align*}
  u_\ka^\Om(x) = \frac{e^{\ii \ka \|z\|}}{4 \pi \|z\|} w_\ka(D,\hat{z};x-z) \left[ 1 + \mc{O}(\|z\|^{-1}) \right].
\end{align*}
Applying the far-field expansion yields \eqref{eq:23a}.

The proof is complete. 
\end{proof}

\subsection{Determination of the location}\label{sect:location}

We are now in a position to present the gesture recognition algorithm of recovering $\Omega=D+z$ by using $u_\ka^\Om|_\Gamma$. In the first step, we shall determine the location of $z$ of $\Om$. We shall achieve this by using an incident point wave with a relatively small wavenumber $\ka$. The following result shall be of importance in designing our algorithm. 

\begin{thm} \label{lem:small}
  Let $D$ be an impenetrable soft scatterer. Let $w_\kappa(D,\hat z; x)$ and $w_\kappa^\infty(D, \hat z; \hat x)$ be defined in Section~\ref{sect:soundsoft}, then there holds
  \begin{align}
    \label{eq:7}
    \lim_{\ka \to +0} w_\ka^\infty(D,\hat{z};\hat{x}) = \al(D)
  \end{align}
uniformly for all $\hat{z} \in \BS^2$ and $\hat{x} \in \BS^2$, where $\al(D)$ is a constant depending only on $D$.
\end{thm}
\begin{proof}
Similar to the proof of Theorem~\ref{thm:point-plane}, we assume that $\partial D$ is $C^2$-continuous so that we can make use of the Brakhage-Werner trick. Denote $\mc{A}_\ka := (\mc{K}_\ka^D + \ii \eta \mc{S}_\ka^D)$ and $\mc{B}_\ka := \left( \frac{I}{2} + \mc{K}_\ka^D + \ii \eta \mc{S}_\ka^D \right)^{-1}$. Since $\mc{A}_\ka, \mc{B}_\ka$ are uniformly bounded as $\ka \to +0$ and $\mc{B}_\ka \to \mc{B}_0$ as $\ka \to +0$, and $w_\ka^{\rm in}(\hat{z};x) = w_0^{\rm in}(\hat{z};x) + \mc{O}(\ka), x \in \pa D$ uniformly for all $\hat{z} \in \BS^2$, we have
  \begin{align*}
    \mc{B}_\ka w_\ka^{\rm in} &= \mc{B}_0 w_0^{\rm in} + \left( \mc{B}_\ka w_0^{\rm in} - \mc{B}_0 w_0^{\rm in} \right) + \left( \mc{B}_\ka w_\ka^{\rm in} - \mc{B}_\ka w_0^{\rm in} \right) = \vp_0 + \psi_\ka,
  \end{align*}
where $\vp_0 := \mc{B}_0 w_0^{\rm in}$ and $\psi_\ka \to 0$ as $\ka \to +0$. 

For $y \in \pa D$ and $\|x\| \to +\infty$ we have the expansion,
\begin{align*}
 \|x-y\| = \|x\| - \hat{x} \cdot y + \mc{O}(\|x\|^{-1})
\end{align*}
uniformly for all $\hat{x} \in \BS^2$. For $y \in \pa D$, $\|x\| \to +\infty$ and $\ka \to +0$, we have
\begin{align}
  e^{\ii \ka \|x-y\|} &= e^{\ii \ka \|x\|} \left[ 1 + \mc{O}(\ka) \right], \notag \\[5pt]
    \Phi_\ka(x,y) &= \dfrac{e^{\ii \ka \|x\|}}{4 \pi \|x\|} \left[ 1 + \mc{O}(\|x\|^{-1}) + \mc{O}(\ka) \right], \label{eq:36} \\[5pt]
    \dfrac{\pa \Phi_\ka(x,y)}{\pa \nu(y)} &= \dfrac{e^{\ii \ka \|x\|}}{4\pi \|x\|} \left[ \mc{O}(\|x\|^{-1}) + \mc{O}(\ka) \right]. \notag
\end{align}
It then follows that
\begin{align*}
  \mc{A}_\ka \vp_0 & = \frac{e^{\ii \ka \|x\|}}{\|x\|} \left[ \al(D) + \mc{O}(\|x\|^{-1}) + \mc{O}(\ka) \right], \\
  \mc{A}_\ka \psi_\ka &= \frac{e^{\ii \ka \|x\|}}{\|x\|} \xi_\ka,
\end{align*}
where
\begin{align*}
  \al(D) = \frac{1}{4\pi} \int_{\pa D} \vp_0(y) \, {\rm d}s(y)
\end{align*}
is a constant depending only on $D$ and
\begin{align*}
  \xi_\ka &= \frac{1}{4\pi} \int_{\pa D} \left[ 1 + \mc{O}(\|x\|^{-1}) + \mc{O}(\ka) \right] \psi_\ka(y) \, {\rm d}s(y) \notag \\
  &+ \frac{\ii \eta}{4\pi} \int_{\pa D} \left[ \mc{O}(\|x\|^{-1}) + \mc{O}(\ka) \right] \psi_\ka(y) \, {\rm d}s(y) \to 0
\end{align*}
as $\ka \to +0$. Hence
\begin{align*}
  w_\ka(D,\hat{z};x) &= \mc{A}_\ka \mc{B}_\ka w_\ka^{\rm in} = \mc{A}_\ka (\vp_0 + \psi_\ka) \notag \\
  &= \frac{e^{\ii \ka \|x\|}}{\|x\|} \left[ \al(D) + \mc{O}(\|x\|^{-1}) + \mc{O}(\ka) + \xi_\ka \right]
\end{align*}
and \eqref{eq:7} follows from the far-field expansion.

The proof is complete. 
\end{proof}

\medskip

Similar to Theorem~\ref{lem:small}, we have the following result for the scattering from a medium scatterer $(\Omega, n_\Om)$.

\begin{thm} \label{lem:small-medium}
 Let $(D, n_D)$ be a penetrable medium scatterer. Let $w_\kappa(D,\hat z; x)$ and $w_\kappa^\infty(D, \hat z; \hat x)$ be defined in Section~\ref{sect:medium}, then there holds
  \begin{align}
    \label{eq:35}
    w_\ka^\infty(D,\hat{z};\hat{x}) = \ka^2 \left[ \al(D,n_D) + \mc{O}(\ka) \right],
  \end{align}
  where $\al(D)$ is a constant depending only on $D$ and $n_D$.
\end{thm}
\begin{proof}
  Using the expansion \eqref{eq:36} and $e^{-\ii \ka \hat{z} \cdot y} = 1 + \mc{O}(\ka)$, we deduce
  \begin{align}
    \mc{T}_\ka^D w_\ka^{\rm in}(\hat{z};x) &= \ka^2 \frac{e^{\ii \ka \|x\|}}{4 \pi \|x\|} \int_D \left[1+\mc{O}(\|x\|^{-1}) + \mc{O}(\ka) \right] [1+\mc{O}(\ka)] m_D(y) \, {\rm d}y \notag \\
    &= \frac{e^{\ii \ka \|x\|}}{4 \pi \|x\|} \left[ \ka^2 \al(D,n_D) + \mc{O}(\|x\|^{-1}) + \mc{O}(\ka^3) \right]     \label{eq:37}
  \end{align}
  as $\|x\| \to \infty$ and $\ka \to +0$, where $\al(D,n_D) = \int_D m_D(y) \, {\rm d}y$. It is also easy to see $\|T_\ka^D\|_{\mathcal{L}(L^2(D), H^2(D))} = \mc{O}(\ka^2)$ as $\ka \to +0$. Hence it follows from \eqref{eq:33} and \eqref{eq:37} that
  \begin{align*}
    w_\ka(D,\hat{z};x) = \frac{e^{\ii \ka \|x\|}}{4 \pi \|x\|} \left[ \ka^2 \al(D) + \mc{O}(\|x\|^{-1}) + \mc{O}(\ka^3) \right],
  \end{align*}
 which readily implies \eqref{eq:35} by using the far-field expansion.
 
 The proof is complete. 
\end{proof}

With the help of Theorems~\ref{lem:small} and \ref{lem:small-medium}, we are ready to present the scheme of recovering the location point $z$ of the scatterer $\Omega=D+z$. Without loss of generality, we assume that the scatterer is located in the half-space $\RR^3_+:=\{x=(x^1,x^2,x^3):x^1>0 \}$. Let the measurement surface $\Gamma$ be a bounded set in the $x^2x^3$-plane. Let $\widetilde{z} \in \RR^3_+$ be an arbitrary sampling point contained in a bounded sampling region $S \subset \RR^3_+$. In view of Theorems \ref{thm:point-plane} and \ref{lem:small}, we propose the following indicator functional for the determination of the location of $\Om=D+z$:
\begin{align}
  \label{eq:34}
  I_\ka(D,z;\widetilde z) := \frac{\left| \big\langle u_\ka(D,z; \cdot), \mathring{u}_\ka(\widetilde z;\cdot) \big\rangle_{L^2(\Ga)} \right|}{\left\| u_\ka(D,z;\cdot) \right\|_{L^2(\Ga)} \left\| \mathring{u}_\ka (\widetilde z;\cdot) \right\|_{L^2(\Ga)}}, \quad \widetilde z \in S,
\end{align}
where $u_\ka(D,z;x)=u_\ka^\Om(x), x \in \Ga$ is the scattered wave field (cf. \eqref{eq:4} and \eqref{eq:10}) measured on the surface $\Ga$ due to the scatterer $\Om$ and the incident field $u_\ka^{\rm in}$ in \eqref{eq:2}; and the test function
\begin{align*}
  \mathring{u}_\ka(\widetilde z;x) := \frac{e^{\ii \ka \|\widetilde z\|}}{4 \pi \|\widetilde z\|} \frac{e^{\ii \ka \|x-\widetilde z\|}}{\|x-\widetilde z\|}.
\end{align*}
If the measurement data are phaseless, then we modify the indicator functional as
\begin{align}
  \label{eq:24}
  I_\ka(D,z;\widetilde z) := \frac{\left| \left\langle |u_\ka(D,z; \cdot)|, | \mathring{u}_\ka(\widetilde z;\cdot) | \right\rangle_{L^2(\Ga)} \right|}{\left\| u_\ka(D,z;\cdot) \right\|_{L^2(\Ga)} \left\| \mathring{u}_\ka (\widetilde z; \cdot) \right\|_{L^2(\Ga)}}, \quad \widetilde z \in S.
\end{align}

We can show the following indicating behavior of the functionals introduced in \eqref{eq:34} and \eqref{eq:24}, which can help us to find the location point $z$. 

\medskip

\begin{thm} \label{thm:location}
 Let $\al(D)$ be given in Theorem \ref{lem:small} (if $D$ is an impenetrable soft scatterer) or Theorem \ref{lem:small-medium} (if $D$ is a penetrable medium scatterer) and assume $\al(D) \neq 0$ for all $D \in \ms{A}$. Then we have the following asymptotic expansion
  \begin{align}
    \label{eq:9}
    \lim_{\ka \to 0} I_\ka(D,z; \widetilde z) = \mathring{I}_0(z;\widetilde z) \left[ 1 + \mc{O}(\|z\|^{-1}) \right], \quad |z| \to \infty
  \end{align}
uniformly for all $D \in \ms{A}, \hat{z} \in \BS^2$ and $\widetilde z \in S$, where
\begin{align*}
  \mathring{I}_\ka(z;\widetilde z) = \frac{\left| \big\langle \mathring{u}_\ka(z;\cdot), \mathring{u}_\ka(\widetilde z;\cdot) \big\rangle_{L^2(\Ga)} \right|}{\left\| \mathring{u}_\ka(z,\cdot) \right\|_{L^2(\Ga)} \left\| \mathring{u}_\ka (\widetilde z;\cdot) \right\|_{L^2(\Ga)}}, \quad \widetilde z \in S
\end{align*}
if $I_\ka$ is given by \eqref{eq:34}; or
\begin{align*}
  \mathring{I}_\ka(z;\widetilde z) = \frac{\left| \big\langle |\mathring{u}_\ka(z;\cdot)|, | \mathring{u}_\ka(\widetilde z;\cdot)| \big\rangle_{L^2(\Ga)} \right|}{\left\| \mathring{u}_\ka(z;\cdot) \right\|_{L^2(\Ga)} \left\| \mathring{u}_\ka (\widetilde z;\cdot) \right\|_{L^2(\Ga)}}, \quad \widetilde z \in S,
\end{align*}
if $I_\ka$ is given by \eqref{eq:24}. The unique maximum of $\mathring{I}_\ka(z;\widetilde z)$ is obtained at $\widetilde z=z$ with maximal value $1$.
\end{thm}
\begin{proof}
We first consider the case when $D$ is an impenetrable scatterer and $I_\ka$ is given by \eqref{eq:34}. By Theorems \ref{thm:point-plane} and \ref{lem:small}, we have
  \begin{align}
    \lim_{\ka \to +0} u_\ka(D,z;x) &= \mathring{u}_0(z;x) \left[ \al(D) + \mc{O}(\|z\|^{-1}) \right] \left[ 1 + \mc{O}(\|z\|^{-1}) \right] \notag \\
    &= \mathring{u}_0(z;x) \al(D) \left[ 1 + \mc{O}(\|z\|^{-1}) \right], \quad \|z\| \to +\infty, \label{eq:12}
  \end{align}
uniformly for all $\hat{z} \in \BS^2$ and $x \in S$. Plugging \eqref{eq:12} into \eqref{eq:34} yields \eqref{eq:9} immediately.

It follows from the Cauchy-Schwarz inequality that $|\mathring{I}_\ka(z; \widetilde z)| \leq 1$ for all $\widetilde z \in \RR^3_+$ where the equality holds only when $\mathring{u}_\ka(z;\cdot)$ and $\mathring{u}_\ka(\widetilde z;\cdot)$ are constant multiples of each other. Clearly this occurs only if $\widetilde z=\pm z$, and furthermore since $z \in \RR^3_+$ and $\widetilde z \in \RR^3_+$, we must have that $\widetilde z=z$.

The case when $D$ is a penetrable medium scatterer and $I_\ka$ is given by \eqref{eq:24} follows from similar arguments. The proof is complete. 
\end{proof}

From Theorem \ref{thm:location} one can expect the maximum of $I_\ka(D,z;\widetilde z)$ will be achieved at $\widetilde z \approx z$. Furthermore one can expect that the maximum point is unique if $\ka$ is sufficiently small, which means the approximate location can be found efficiently using local optimization algorithms.

\subsection{Gesture recognition}\label{sect:shape}

After the determination of the approximate location $\mathring{z}={\rm arg\, max}_{\widetilde z}\, I_\ka(D,z; \widetilde z)$ of the gesture $\Omega=z+D$ in Section~\ref{sect:location}, we proceed to the determination of the shape $D$. To that end, we shall make use of another incident field with a wavenumber $\ka \simeq 1$. In view of Theorem \ref{thm:point-plane}, we propose the following indicator functional,
\begin{align}
  \label{eq:41}
  J_\ka(D_i,D_j;z,\mathring{z}) := \frac{\left| \big\langle u_\ka(D_i,z; \cdot), \hat{u}_\ka(D_j,\mathring{z};\cdot) \big\rangle_{L^2(\Ga)} \right|}{\left\| u_\ka(D_i,z;\cdot) \right\|_{L^2(\Ga)} \left\| \hat{u}_\ka (D_j,\mathring{z};\cdot) \right\|_{L^2(\Ga)}}
\end{align}
for $D_i, D_j \in \ms{A}$, where $u_\ka(D_i, z; \cdot):=u_\ka^{D_i+z}(\cdot)$ and the test function
\begin{align}\label{eq:data}
  \hat{u}_\ka(D_j,\mathring{z};x) := \frac{e^{\ii \ka \|\mathring{z}\|}}{4 \pi \|\mathring{z}\|} \frac{e^{\ii \ka \|x-\mathring{z}\|}}{\|x-\mathring{z}\|} w_\ka^\infty(D_j,\hat{\mathring{z}};\widehat{x-\mathring{z}}),
\end{align}
where $w_\ka^\infty$ is defined, respectively, in Sections~\ref{sect:soundsoft} and \ref{sect:medium}, corresponding to the cases when $D$ is impenetrable and penetrable. If the measurement data are phaseless, then we modify the indicator function as follows,
\begin{align}
  \label{eq:28}
  J_\ka(D_i,D_j;z,\mathring{z}) := \frac{\left| \big\langle | u_\ka(D_i,z; \cdot) |, | \hat{u}_\ka(D_j,\mathring{z};\cdot) | \big\rangle_{L^2(\Ga)} \right|}{\left\| u_\ka(D_i,z,\cdot) \right\|_{L^2(\Ga)} \left\| \hat{u}_\ka (D_j,\mathring{z};\cdot) \right\|_{L^2(\Ga)}}.
\end{align}

The identification of the shape $D$ is based on the following theorem.\medskip

\begin{thm}\label{thm:regular}
  Assume there exists a constant $c_0\in\mathbb{R}_+$ such that $\|u_\ka(D_i,z;\cdot)\|_{L^2(\Ga)} > c_0$ for all $D_i \in \ms{A}$. Then for any sufficiently small $\ve \in\mathbb{R}_+$ there exist $R\in\mathbb{R}_+$ and $\de\in\mathbb{R}_+$ such that if $\|z\| > R$ and $\|z - \mathring{z}\| < \de$,
  \begin{align*}
    \left| J_\ka(D_i,D_j;z,\mathring{z}) - \hat{J}_\ka(D_i,D_j;z) \right| < \ve, \quad \forall\, D_i,D_j \in \ms{A},
  \end{align*}
 where
\begin{align*}
  \hat{J}_\ka(D_i,D_j;z) := \frac{\left| \big\langle u_\ka(D_i,z; \cdot), u_\ka(D_j,z;\cdot) \big\rangle_{L^2(\Ga)} \right|}{\left\| u_\ka(D_i,z;\cdot) \right\|_{L^2(\Ga)} \left\| u_\ka (D_j,z;\cdot) \right\|_{L^2(\Ga)}}
\end{align*}
if $J_\ka$ is given by \eqref{eq:41}; or
\begin{align*}
  \hat{J}_\ka(D_i,D_j;z) = \frac{\left| \big\langle | u_\ka(D_i,z; \cdot) |, | u_\ka(D_j,z;\cdot) | \big\rangle_{L^2(\Ga)} \right|}{\left\| u_\ka(D_i,z;\cdot) \right\|_{L^2(\Ga)} \left\| u_\ka (D_j,z;\cdot) \right\|_{L^2(\Ga)}}
\end{align*}
if $J_\ka$ is given by \eqref{eq:28}.

If we further assume $u_\ka(D_i,z;\cdot)|_\Ga$ and $u_\ka(D_j,z;\cdot)|_\Ga$ are linearly independent for all $D_i,D_j \in \ms{A}$ with $i\neq j$ if $J_\ka$ is given by \eqref{eq:41} (or $|u_\ka(D_i,z;\cdot)|_\Ga|$ and $|u_\ka(D_j,z;\cdot)|_\Ga|$ are linearly independent for all $D_i,D_j \in \ms{A}$ in the case when $J_\ka$ is given by \eqref{eq:28}), then there exists $R>0$ and $\de>0$ such that if $\|z\| > R$ and $\|z - \mathring{z}\| < \de$, then there holds
\begin{align*}
  J_\ka(D_i,D_i;z,\mathring{z}) > J_\ka(D_i,D_j;z,\mathring{z}), \quad \forall\, i \neq j.
\end{align*}
\end{thm}

\begin{proof}
In the following, we only consider the case when $D$ is impenetrable and $J_\ka$ is given by \eqref{eq:41}. The case when $D$ is a penetrable medium scatterer and/or $J_\ka$ is given by \eqref{eq:28} can be proven in a completely similar manner.

Let $\ve>0$ be sufficiently small and fixed. By Theorem \ref{thm:point-plane} there exists $R>0$ such that if $\|z\| > R$,
  \begin{align}
    \label{eq:15}
    \left\| u_\ka(D_j,z;\cdot) - \hat{u}_\ka(D_j,z;\cdot) \right\|_{L^2(\Ga)} < \frac{\ve}{8} \left\| u_\ka(D_j,z;\cdot) \right\|_{L^2(\Ga)}, \ \forall j\in\{1,2,\ldots,N\}. 
  \end{align}
For given $z$ such that $\|z\|>R$ it follows from the analyticity of $w_\ka^\infty(D,\hat{\tilde z};\widehat{x-\tilde z})$ in $\tilde z$ that there exists $\de>0$ such that if $\|z-\mathring{z}\|<\de$, one has
\begin{align}
  \label{eq:16}
  \left\| \hat{u}_\ka(D_j,z;\cdot) - \hat{u}_\ka(D_j,\mathring{z};\cdot) \right\|_{L^2(\Ga)} < \frac{\ve}{8} \left\| u_\ka(D_j,z;\cdot) \right\|_{L^2(\Ga)}, \ \forall\,j\in\{1,2,\ldots,N\}. 
\end{align}
Combining \eqref{eq:15} and \eqref{eq:16} yields, for any given $z,\mathring{z} \in S$ such that $\|z\|>R$ and $\|z-\mathring{z}\|<\de$, there holds
\begin{align}
  \label{eq:17}
  \hat{u}_\ka(D_j,\mathring{z};x) = u_\ka(D_j,z;x) + v_\ka(D_j,x), \quad x \in \Ga,
\end{align}
with 
\[
\| v_\ka(D_j,\cdot) \|_{L^2(\Ga)} < \frac{\ve}{4} \| u_\ka(D_j,z;\cdot) \|_{L^2(\Ga)}\ \ \mbox{for all}\ \ D_j \in \ms{A}.
\]
Using the Cauchy-Schwarz inequality, one has
\begin{align}
  \label{eq:19}
  \left| \big\langle u_\ka(D_i,z; \cdot), \hat{u}_\ka(D_j,\mathring{z};\cdot) \big\rangle_{L^2(\Ga)} \right| =
  \left| \big\langle u_\ka(D_i,z; \cdot), u_\ka(D_j,z;\cdot) \big\rangle_{L^2(\Ga)} \right| + \de,
\end{align}
where $|\de| < \frac{\ve}{4} \| u_\ka(D_i,z;\cdot) \|_{L^2(\Ga)} \| u_\ka(D_j,z;\cdot) \|_{L^2(\Ga)}$. It follows from \eqref{eq:17} that
\begin{align}
  \label{eq:20}
  \left\| \hat{u}_\ka (D_j,\mathring{z};\cdot) \right\|_{L^2(\Ga)} =
  \left\| u_\ka (D_j,z;\cdot) \right\|_{L^2(\Ga)} ( 1 + \si)
\end{align}
with $|\si| < \frac{\ve}{4}$. Substituting \eqref{eq:19} and \eqref{eq:20} into \eqref{eq:41} yields
\begin{align*}
  J_\ka(D_i,D_j;z,\mathring{z}) - \hat{J}_\ka(D_i,D_j;z) &= \frac{-\si}{1+\si} \hat{J}_\ka(D_i,D_j;z) \\
  &+ \frac{\de}{\| u_\ka(D_i,z;\cdot) \|_{L^2(\Ga)} \| u_\ka(D_j,z;\cdot) \|_{L^2(\Ga)}(1+\si)}.
\end{align*}
Noting that $|\hat{J}_\ka(D_i,D_j;z)| \leq 1$, we obtain
\begin{align*}
  \left| J_\ka(D_i,D_j;z,\mathring{z}) - \hat{J}_\ka(D_i,D_j;z) \right| < 2 |\si| + \frac{\ve}{2} < \ve.
\end{align*}

The statement in the second part of the theorem can be readily shown by using the Cauchy-Schwarz inequality. The proof is complete. 
\end{proof}

By using Theorem~\ref{thm:regular}, the identification of the shape can be proceeded as follows. One first collects the measurement data $u_\ka^\Om|_\Ga$ (resp. $|u_\ka^\Om|_\Ga|$), and then compute the indicator functional \eqref{eq:41} (resp. \eqref{eq:28}), by taking $D_i=D$ and running the trial shape $D_j$ through all the dictionary shapes in $\mathscr{A}$. According to Theorem~\ref{thm:regular}, one readily sees that only when the trial shape $D_j=D$, the indicator functional achieves its maximum value (being approximately $1$). We note that in order to calculate the indicator functionals \eqref{eq:41} or \eqref{eq:28}, one needs the far-field data of all the dictionary scatterers $D_j\in\mathscr{A}$ corresponding to incident plane waves (cf. \eqref{eq:data}). It is remarked that those dictionary scattering data can be captured and saved beforehand in the gesture recognition device. 

\section{Numerical experiments}
In this section we describe the numerical implementation and conduct numerical experiments to test the effectiveness and efficiency of the recognition method.

Without loss of generality, we assume that the target scatterer is given by $\Omega:=D_i+z_0$ for some $D_i\in\mathscr{A}$ and $z_0\in \RR^3_+$. Let the measurement surface $\Ga$ be a square in the $x^2x^3$-plane and centred at the origin. The measurement data $v(\ka_1,D_i,z_0;x)$ and $v(\ka_2,D_i,z_0;x)$ are taken at a uniformly distributed grid points on $\Ga$. The $L^2$-inner product and the $L^2$-norm on $\Ga$ in the indicator functionals \eqref{eq:34} and \eqref{eq:21} are approximately computed using the composite Trapezoidal rule on the grid.

Using polar coordinates $(\theta,\vp)\in \mathbb{S}^2$, we let $(\te^{\rm in}_j, \vp^{\rm in}_k)$ be a uniform mesh of the incident angles from $\BS^2_+=[0,\pi] \times [-\pi/2,\pi/2]$, and $(\te_m, \vp_n)$ be a uniform mesh of observation angles from $\BS^2_-=[0,\pi] \times [\pi/2,3\pi/2]$. We compute and save $w^\infty(\ka_2,D_i,\hat{z}_{jk};\hat{x}_{mn})$ for each $D_i \in \ms{A}$ and
\begin{align*}
  \hat{z}_{jk} &= \left[ \sin(\te^{\rm in}_j)\cos(\vp^{\rm in}_k), \sin(\te^{\rm in}_j)\sin(\vp^{\rm in}_k), \cos(\te^{\rm in}_j) \right], \\[3pt]
\hat{x}_{mn} &= \left[ \sin(\te_m)\cos(\vp_n), \sin(\te_m)\sin(\vp_n), \cos(\te_m) \right].
\end{align*}
For fixed $D \in \ms{A}$, $\hat{z} \in \BS^2_+$ and $\hat{x}_{mn}$, we let $w^\infty_z(\ka,D;\hat{x}_{mn})$ be the bilinear interpolation of $w^\infty(\ka,D_i,\hat{z}_{jk};\hat{x}_{mn})$ at $\hat{z}$. For fixed $x \in \Ga$, let $w^\infty_{zx}(\ka,D)$ be the bilinear interpolation of $w^\infty_z(\ka,D;\hat{x}_{mn})$ at $\widehat{x-z}$. Then we have the approximation $w^\infty(\ka,D,\hat{z};\widehat{x-z}) \approx w^\infty_{zx}(\ka,D)$ and the function $\hat{u}$ in \eqref{eq:41} can be computed efficiently from precomputed data saved in the gesture recognition device.

For the numerical experiments, the admissible class/dictionary consists of six scatterers as shown in Figure \ref{fig:admissible-class}. Each scatterer is composed of four unit cubes.
\begin{figure}[t]
  \centering
  \begin{subfigure}[b]{0.2\textwidth}
    \includegraphics[width=\textwidth]{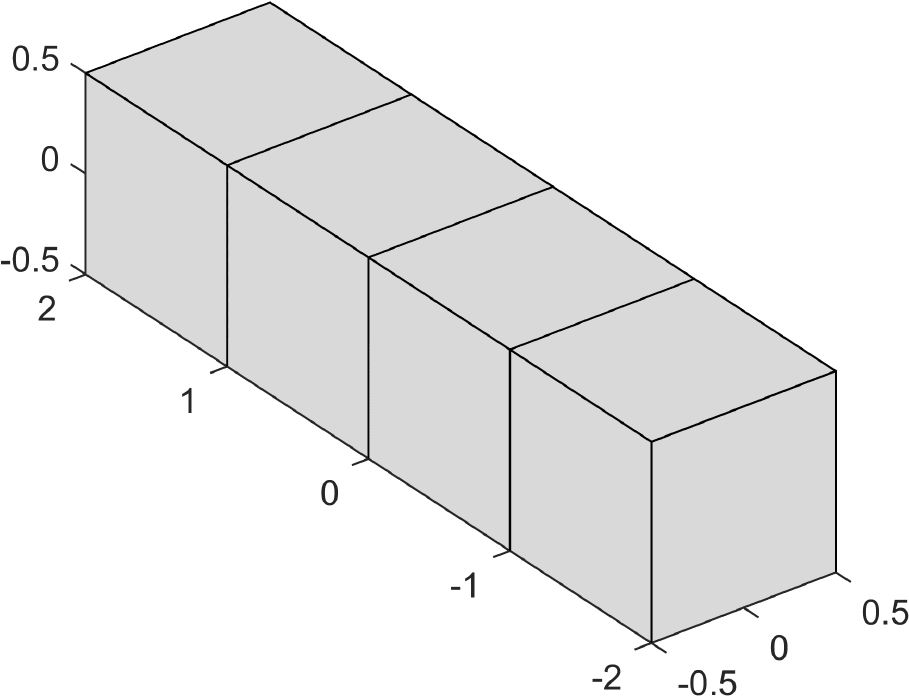}
    \caption{$D_1$}
  \end{subfigure}
\hspace{1cm}
  \begin{subfigure}[b]{0.2\textwidth}
    \includegraphics[width=\textwidth]{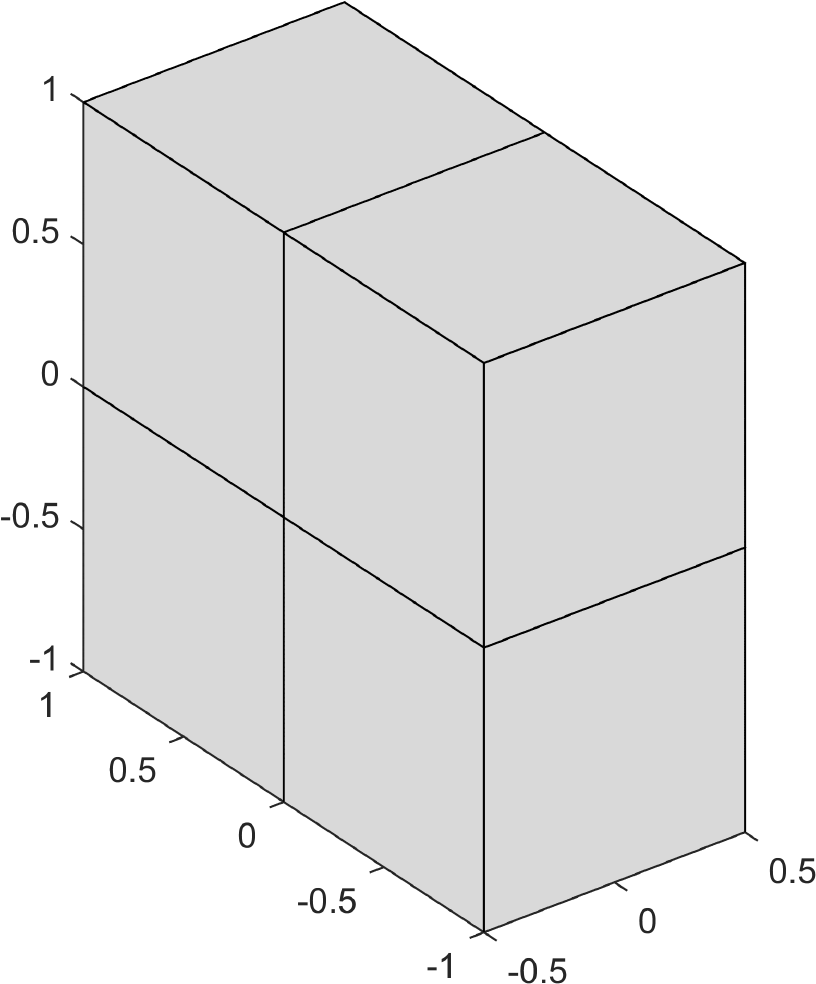}
    \caption{$D_2$}
  \end{subfigure}
\hspace{1cm}
  \begin{subfigure}[b]{0.2\textwidth}
    \includegraphics[width=\textwidth]{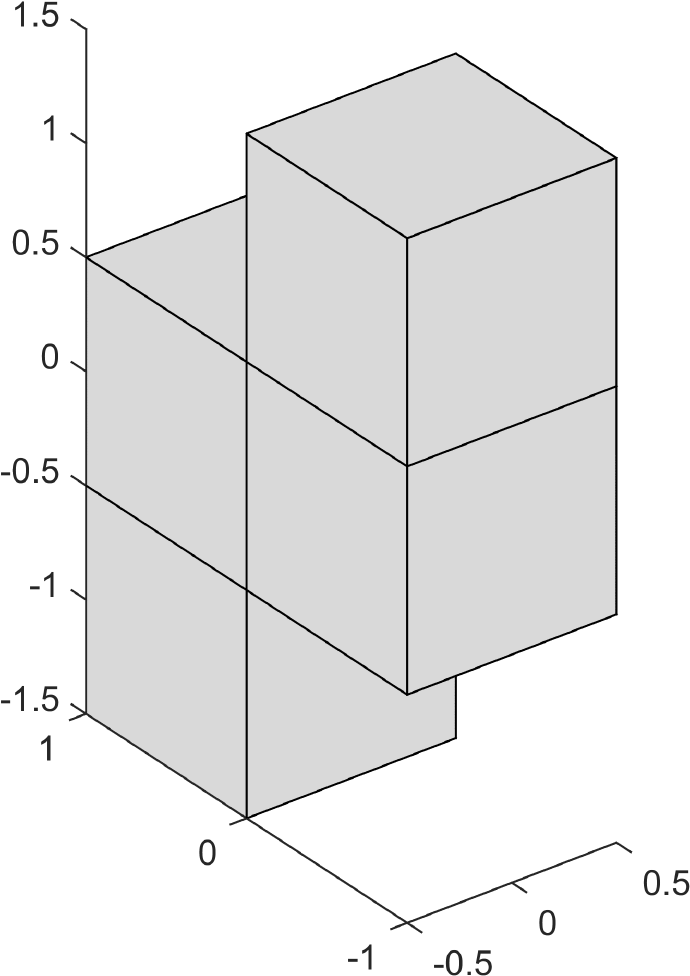}
    \caption{$D_3$}
  \end{subfigure}
\\[20pt]
  \begin{subfigure}[b]{0.2\textwidth}
    \includegraphics[width=\textwidth]{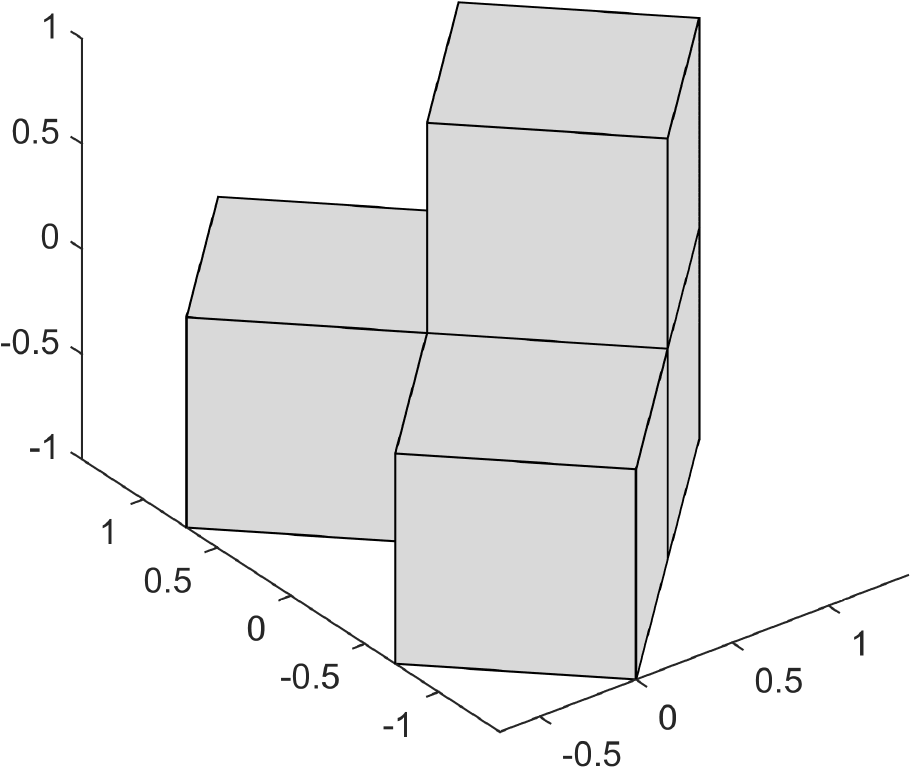}
    \caption{$D_4$}
  \end{subfigure}
\hspace{1cm}
  \begin{subfigure}[b]{0.2\textwidth}
    \includegraphics[width=\textwidth]{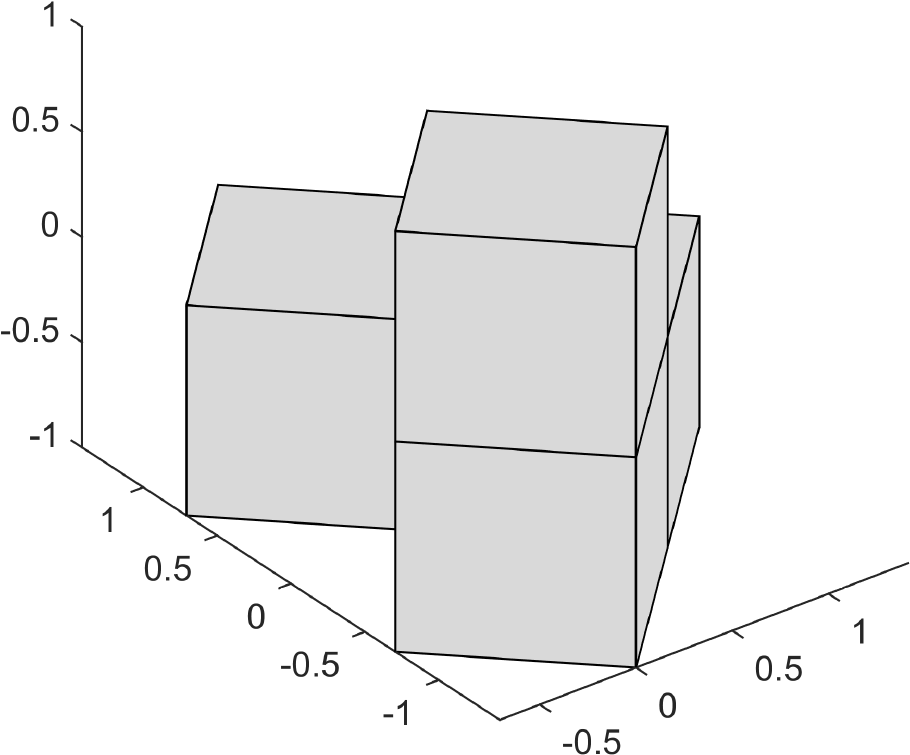}
    \caption{$D_5$}
  \end{subfigure}
\hspace{1cm}
  \begin{subfigure}[b]{0.2\textwidth}
    \includegraphics[width=\textwidth]{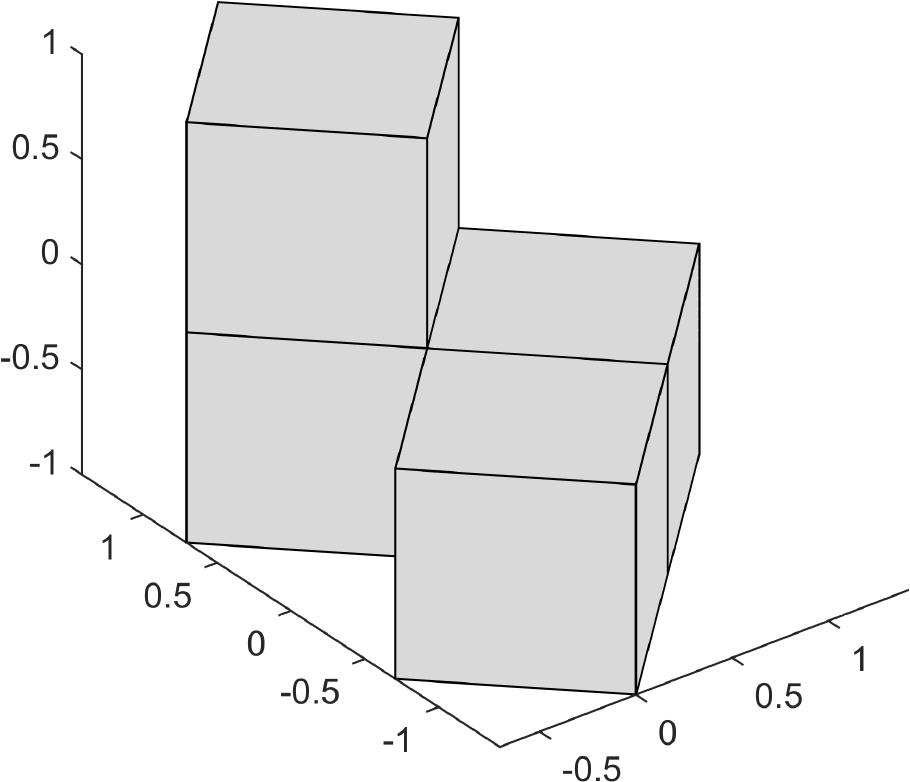}
    \caption{$D_6$}
  \end{subfigure}
  \caption{The dictionary consists of six scatterers. Each scatterer is comprised with four unit cubes.}
  \label{fig:admissible-class}
\end{figure}

In the numerical experiments in what follows, we let $z_0=[50,0,0]$ be fixed. The wavelength for determining the location is set to be $\la_1:=2\pi/\ka_1=100$ and the wavelength for the shape identification is set to be $\la_2=:=2\pi/\ka_2=1$. The measurement surface $\Ga$ is a square in the $x^2x^3$-plane, centred at the origin and has a side length $20$. The measurement data is taken at a $32 \times 32$ uniform mesh on $\Ga$. The far-field data are computed and saved at a $180 \times 180$ uniform mesh for the incident angles and $180 \times 180$ uniform mesh for the observation angles. The sampling region for determining the location is set as $S=[0,100] \times [-100,100] \times [-100,100]$. The optimization of the indicator functions \eqref{eq:34} or \eqref{eq:24} is performed using the Matlab function \texttt{fmincon} with the \texttt{sqp} algorithm, without using the gradient information, and with an initial guess at $[10,0,0]$.
\begin{table}[t]
  \centering
  \begin{tabular}[c]{|c|c|c|c|c|c|c|}
    \hline
    $ $ & $D_1$ & $D_2$ & $D_3$ & $D_4$ & $D_5$ & $D_6$ \\
    \hline
    $\mathring{z}_0^1$ & 50.0279 & 50.0048 & 50.0117 & 50.2805 & 50.1344 & 50.1301 \\
    \hline
    $\mathring{z}_0^2$ & -0.0130 & -0.0011 & 0.0020 & -0.0029 & -0.1435 & 0.1449 \\
    \hline
    $\mathring{z}_0^3$ & -0.0010 & -0.0018 & -0.0025 & -0.2181 & -0.1992 & -0.1944 \\
    \hline    
    $|\mathring{z}_0 - z_0|$ & 0.0308 & 0.0053 & 0.0121 & 0.3553 & 0.2799 & 0.2752 \\
    \hline
  \end{tabular}
  \caption{Approximate location $\mathring{z}_0$ for each scatterer in the dictionary using noise-free measurement data with phase. All $D_i \in \ms{A}$ are impenetrable soft scatterers.}
  \label{tab:location_phase}
\end{table}

\subsection{Impenetrable scatterers}

We first consider the case when all $D_i \in \ms{A}$ are impenetrable soft scatterers. In the first step of the identification, we determine the approximate location of the scatterer as the maximizer of the indicator function $I_\ka$ (cf. \eqref{eq:34} and \eqref{eq:24}). The coordinates and the distance from the exact location are presented in Table \ref{tab:location_phase} for each scatterer in the dictionary.
\begin{table}[t]
  \centering
  \begin{tabular}[c]{|c|c|c|c|c|c|c|}
    \hline
    $ $ & $D_1$ & $D_2$ & $D_3$ & $D_4$ & $D_5$ & $D_6$ \\
    \hline
    $D_1$ & \bf{0.9986} & 0.9694 & 0.9574 & 0.7930 & 0.9217 & 0.9231 \\
    \hline
    $D_2$ & 0.9690 & \bf{1.0000} & 0.9914 & 0.8271 & 0.9499 & 0.9507 \\
    \hline
    $D_3$ & 0.9571 & 0.9912 & {\bf 0.9998} & 0.8198 & 0.9443 & 0.9383 \\
    \hline
    $D_4$ & 0.8465 & 0.8899 & 0.8811 & \bf{0.9859} & 0.8984 & 0.8992 \\
    \hline
    $D_5$ & 0.9363 & 0.9672 & 0.9603 & 0.8139 & \bf{0.9872} & 0.9680 \\
    \hline
    $D_6$ & 0.9366 & 0.9666 & 0.9511 & 0.8161 & 0.9685 & \bf{0.9876} \\
    \hline
  \end{tabular}
  \caption{The $j$-th row and $i$-th column of the number array gives the value of $J_\ka(D_i,D_j,z_0;\mathring{z}_0)$ using noise free measurement data with phase. All $D_i \in \ms{A}$ are impenetrable scatterers.}
  \label{tab:identification_phase}
\end{table}

Next we compute the value of the indicator function $J_\ka$ (cf. \eqref{eq:41} and \eqref{eq:28}) at the approximate location found in table \ref{tab:location_phase}. The results are listed in Table \ref{tab:identification_phase}, where the value of $J_\ka(D_i,D_j,z_0;\mathring{z}_0)$ is listed in the $j$-th row and $i$-th column of the array. The maximum value in each row is marked in bold face. Clearly the maximum is obtained at $i=j$ and the scatterer is identified in each case. The whole process of a recognition takes less than one second in a personal computer (excluding the time needed to precompute the far-field pattern and save the data in the disk).

Alternatively, we can use phaseless measurement data in both the first and the second step of the identification.
\begin{table}[t]
  \centering
  \begin{tabular}[c]{|c|c|c|c|c|c|c|}
    \hline
    $ $ & $D_1$ & $D_2$ & $D_3$ & $D_4$ & $D_5$ & $D_6$ \\
    \hline
    $\mathring{z}_0^1$ & 50.0144 & 50.0351 & 50.0265 & 50.3865 & 50.2216 & 50.2182 \\
    \hline
    $\mathring{z}_0^2$ & -0.0233 & -0.0021 & 0.0009 & -0.0039 & -0.1311 & 0.1337 \\
    \hline
    $\mathring{z}_0^3$ & -0.0035 & -0.0039 & -0.0001 & -0.1402 & -0.2281 & -0.2199 \\
    \hline
    $\|\mathring{z}_0 - z_0\|$ & 0.0276 & 0.0353 & 0.0265 & 0.4112 & 0.3440 & 0.3374 \\
    \hline
  \end{tabular}
  \caption{Approximate location $\mathring{z}_0=(\mathring{z}_0^1,\mathring{z}_0^2,\mathring{z}_0^3)$ for each scatterer in the dictionary using noise-free phaseless measurement data. All $D_i \in \ms{A}$ are impenetrable scatterers.}
  \label{tab:location_phaseless}
\end{table}
The locations found using the phaseless data are listed in Table \ref{tab:location_phaseless}. We observe that the results are equally good compared to the results found with the full data. 
\begin{table}[t]
  \centering
  \begin{tabular}[c]{|c|c|c|c|c|c|c|}
    \hline
    $ $ & $D_1$ & $D_2$ & $D_3$ & $D_4$ & $D_5$ & $D_6$ \\
    \hline
    $D_1$ & {\bf 0.9998} & 0.9731 & 0.9614 & 0.9403 & 0.9502 & 0.9509 \\
    \hline
    $D_2$ & 0.9694 & {\bf 1.0000} & 0.9915 & 0.9727 & 0.9810 & 0.9815 \\
    \hline
    $D_3$ & 0.9581 & 0.9918 & {\bf 1.0000} & 0.9579 & 0.9734 & 0.9691 \\
    \hline
    $D_4$ & 0.9320 & 0.9718 & 0.9578 & \bf{0.9998} & 0.9298 & 0.9306 \\
    \hline
    $D_5$ & 0.9454 & 0.9786 & 0.9688 & 0.9340 & {\bf 0.9999} & 0.9830 \\
    \hline
    $D_6$ & 0.9463 & 0.9786 & 0.9650 & 0.9339 & 0.9830 & {\bf 0.9999} \\
    \hline
  \end{tabular}
  \caption{The $j$-th row and $i$-th column of the number array gives the value of $J_\ka(D_i,D_j,z_0;\mathring{z}_0)$ using noise free phaseless measurement data. All $D_i \in \ms{A}$ are impenetrable soft scatterers.}
  \label{tab:identification_phaseless}
\end{table}
The identification results using phaseless data are presented in Table \ref{tab:identification_phaseless}. Again the results are similar to those found using measurement data with phases.

Since the computation of the indicator functionals involves only algebraic operations on the measured data, it is expected that the method is robust with respect to measurement noise. In fact, the sensitivity of the indicator functionals tends to zero as the grid size of the measurement surface tends to zero if the noise is assumed to be white.
\begin{table}[t]
  \centering
  \begin{tabular}[c]{|c|c|c|c|c|c|c|}
    \hline
    $ $ & $D_1$ & $D_2$ & $D_3$ & $D_4$ & $D_5$ & $D_6$ \\
    \hline
    $D_1$ & {\bf 0.9977} & 0.9737 & 0.9602 & 0.9393 & 0.9497 & 0.9491 \\
    \hline
    $D_2$ & 0.9592 & {\bf 0.9982} & 0.9911 & 0.9724 & 0.9796 & 0.9852 \\
    \hline
    $D_3$ & 0.9504 & 0.9915 & {\bf 0.9989} & 0.9564 & 0.9749 & 0.9703 \\
    \hline
    $D_4$ & 0.9502 & 0.9822 & 0.9754 & {\bf 0.9967} & 0.9438 & 0.9457 \\
    \hline
    $D_5$ & 0.9151 & 0.9750 & 0.9597 & 0.9318 & {\bf 0.9993} & 0.9796 \\
    \hline
    $D_6$ & 0.9129 & 0.9780 & 0.9665 & 0.9298 & 0.9787 & {\bf 0.9993} \\
    \hline
  \end{tabular}
  \caption{The $j$-th row and $i$-th column of the number array gives the value of $J_\ka(D_i,D_j,z_0;\mathring{z}_0)$ using phaseless measurement data plus $5\%$ relative noise. All $D_i \in \ms{A}$ are impenetrable soft scatterers}
  \label{tab:identification_phaseless_noise}
\end{table}
For the $32 \times 32$ mesh of the measurement surface, we are still able to recoginize all the scatterers with the phaseless data plus a $5\%$ relative white noise with uniform distribution in $[-1,1]$. The robustness with respect to the measurement noise increases as the number of measurement point increases. The value of the indicator function $J_\ka$ for one run of the algorithm is shown in Table \ref{tab:identification_phaseless_noise}.

Finally, we consider the case when all or some of $D \in \ms{A}$ are medium scatterers. Table \ref{tab:identification_phaseless_noise_medium} shows the results when all $D \in \ms{A}$ are medium scatterers with $n_D = 4.0$ and Table \ref{tab:identification_phaseless_noise_mixed} shows the results when $D_1,D_2,D_3$ are impenetrable soft scatterers and $D_4,D_5,D_6$ are medium scatterers with $n_D=4.0$. All of the results are obtained with phaseless data plus $5\%$ relative random noise.
\begin{table}[t]
  \centering
  \begin{tabular}[c]{|c|c|c|c|c|c|c|}
    \hline
    $ $ & $D_1$ & $D_2$ & $D_3$ & $D_4$ & $D_5$ & $D_6$ \\
    \hline
    $D_1$ & {\bf 0.9797} & 0.9681 & 0.9610 & 0.9364 & 0.9121 & 0.9421 \\
    \hline
    $D_2$ & 0.9645 & {\bf 0.9991} & 0.9905 & 0.9765 & 0.9474 & 0.9576 \\
    \hline
    $D_3$ & 0.9514 & 0.9918 & {\bf 0.9994} & 0.9641 & 0.9222 & 0.9506 \\
    \hline
    $D_4$ & 0.9089 & 0.9770 & 0.9660 & {\bf 0.9989} & 0.9311 & 0.9309 \\
    \hline
    $D_5$ & 0.9339 & 0.9542 & 0.9164 & 0.9431 & {\bf 0.9987} & 0.8708 \\
    \hline
    $D_6$ & 0.8689 & 0.9370 & 0.9339 & 0.9387 & 0.8527 & {\bf 0.9964} \\
    \hline
  \end{tabular}
  \caption{The $j$-th row and $i$-th column of the number array gives the value of $J_\ka(D_i,D_j,z_0;\mathring{z}_0)$ using phaseless measurement data plus $5\%$ relative noise. All $D_i \in \ms{A}$ are medium scatterers with $n_D = 4.0$}
  \label{tab:identification_phaseless_noise_medium}
\end{table}

\begin{table}[t]
  \centering
  \begin{tabular}[c]{|c|c|c|c|c|c|c|}
    \hline
    $ $ & $D_1$ & $D_2$ & $D_3$ & $D_4$ & $D_5$ & $D_6$ \\
    \hline
    $D_1$ & {\bf 0.9978} & 0.9799 & 0.9820 & 0.9754 & 0.9548 & 0.9561 \\
    \hline
    $D_2$ & 0.9695 & {\bf 0.9967} & 0.9926 & 0.9139 & 0.8916 & 0.8856 \\
    \hline
    $D_3$ & 0.9807 & {\bf 0.9931} & 0.9981 & 0.9253 & 0.8952 & 0.8945 \\
    \hline
    $D_4$ & 0.9614 & 0.9176 & 0.9279 & {\bf 0.9994} & 0.9334 & 0.9336 \\
    \hline
    $D_5$ & 0.9338 & 0.8472 & 0.8768 & 0.9401 & {\bf 0.9978} & 0.8579 \\
    \hline
    $D_6$ & 0.9563 & 0.8808 & 0.8839 & 0.9365 & 0.8675 & {\bf 0.9988} \\
    \hline
  \end{tabular}
  \caption{The $j$-th row and $i$-th column of the number array gives the value of $J_\ka(D_i,D_j,z_0;\mathring{z}_0)$ using phaseless measurement data plus $5\%$ relative noise. Here $D_1,D_2,D_3$ are impenetrable scatterers and $D_4,D_5,D_6$ are medium scatterers with $n_D=4.0$.}
  \label{tab:identification_phaseless_noise_mixed}
\end{table}

\section*{Acknowledgement}

The work was supported by the Startup and FRG grants from Hong Kong Baptist University, Hong Kong RGC General Research Funds, 12302415 and 405513, and the NSF grant of China, No. 11371115.

\end{document}